\documentclass[11pt]{amsart}
\usepackage{mathrsfs,amssymb,amsfonts,amsmath,latexsym}
\newtheorem{theorem}{Theorem}[section]
\newtheorem{lemma}[theorem]{Lemma}
\newtheorem{corollary}[theorem]{Corollary}
\newtheorem{question}[theorem]{Question}
\newtheorem{example}[theorem]{Example}
\theoremstyle{definition}
\newtheorem{definition}[theorem]{Definition}
\newtheorem{proposition}[theorem]{Proposition}
\theoremstyle{remark}

\begin{document}

\title[A class of quotient spaces in strongly topological gyrogroups]
{A class of quotient spaces in strongly topological gyrogroups}

\author{Meng Bao}
\address{(Meng Bao): College of Mathematics, Sichuan University, Chengdu 610064, P. R. China}
\email{mengbao95213@163.com}

\author{Jie Wang}
\address{(Jie Wang): School of mathematics and statistics,
Minnan Normal University, Zhangzhou 363000, P. R. China}
\email{jiewang0104@163.com}

\author{Xiaoquan Xu*}
\address{(Xiaoquan Xu): School of mathematics and statistics,
Minnan Normal University, Zhangzhou 363000, P. R. China}
\email{xiqxu2002@163.com}

\thanks{The authors are supported by the National Natural Science Foundation of China (Nos. 11661057, 12071199) and the Natural Science Foundation of Jiangxi Province, China (No. 20192ACBL20045)\\
*corresponding author}

\keywords{Topological gyrogroups; metrizability; $\omega^{\omega}$-base; Fr\'echet-Urysohn.}
\subjclass[2010]{Primary 54A20; secondary 11B05; 26A03; 40A05; 40A30; 40A99.}

\begin{abstract}
Quotient space is a class of the most important topological spaces in the research of topology. In this paper, we show that if $(G,\tau ,\oplus)$ is a strongly topological gyrogroup with a symmetric neighborhood base $\mathscr U$ at $0$ and $H$ is an admissible subgyrogroup generated from $\mathscr U$, then $G/H$ is first-countable if and only if it is metrizable. Moreover, if $H$ is neutral and $G/H$ is Fr\'echet-Urysohn with an $\omega^{\omega}$-base, then $G/H$ is first-countable. Therefore, we obtain that if $H$ is neutral, then $G/H$ is metrizable if and only if $G/H$ is Fr\'echet-Urysohn with an $\omega^{\omega}$-base. Finally, it is shown that if $H$ is neutral, $\pi\chi(G/H)=\chi(G/H)$ and $\pi\omega (G/H)=\omega (G/H)$.
\end{abstract}

\maketitle
\section{Introduction}
A.A. Ungar has studied the $c$-ball of relativistically admissible velocities with Einstein velocity addition for many years. The Einstein velocity addition $\oplus _{E}$ is given as the following: $$\mathbf{u}\oplus _{E}\mathbf{v}=\frac{1}{1+\frac{\mathbf{u}\cdot \mathbf{v}}{c^{2}}}(\mathbf{u}+\frac{1}{\gamma _{\mathbf{u}}}\mathbf{v}+\frac{1}{c^{2}}\frac{\gamma _{\mathbf{u}}}{1+\gamma _{\mathbf{u}}}(\mathbf{u}\cdot \mathbf{v})\mathbf{u}),$$ where $\mathbf{u,v}\in \mathbb{R}_{c}^{3}=\{\mathbf{v}\in \mathbb{R}^{3}:||\mathbf{v}||<c\}$ and $\gamma _{\mathbf{u}}$ is given by $$\gamma _{\mathbf{u}}=\frac{1}{\sqrt{1-\frac{\mathbf{u}\cdot \mathbf{u}}{c^{2}}}}.$$ During this time, he defined the gyrogroup in \cite{UA2002}. By the definition of a gyrogroup, it is easy to see that every group is a gyrogroup. For more studies about gyrogroups, see \cite{FM, FM1,FM2,UA1988,UA,UA2005}. By the development of gyrogroups, W. Atiponrat \cite{AW} equipped gyrogroups with a topology and posed the concept of topological gyrogroups. A topological gyrogroup, that is, a gyrogroup $G$ is endowed with a topology such that the binary operation $\oplus: G\times G\rightarrow G$ is jointly continuous and the inverse mapping $\ominus (\cdot): G\rightarrow G$, i.e. $x\rightarrow \ominus x$, is also continuous. Then Cai, Lin and He in \cite{CZ} proved that every topological gyrogroup is a rectifiable space. A series of results on topological gyrogroups have been obtained in \cite{AA2010,BZX,BZX2,LF,LF1,LF2}

By the further research of M\"{o}bius gyrogroups, Einstein gyrogroups, and Proper Velocity gyrogroups, Bao and Lin \cite{BL} found that all of them have an open neighborhood base at the identity element $0$ which is invariant under the groupoid automorphism with standard topology. Therefore, they posed the concept of strongly topological gyrogroups. At the same time, they proved that a strongly topological gyrogroup $G$ is feathered if and only if it contains a compact $L$-subgyrogroup $H$ such that the quotient space $G/H$ is metrizable, which implies that every feathered strongly topological gyrogroup is paracompact. Then they \cite{BL1} obtained that every $T_{0}$ strongly topological gyrogroup is completely regular and every strongly topological gyrogroup with a countable pseudocharacter is submetrizable. In \cite{BL2}, it was shown that all locally paracompact strongly topological gyrogroups are paracompact. By the study in \cite{BL3}, Bao and Lin studied the class of (strongly) topological gyrogroups with a non-measurable cardinality and they proved that every submaximal topological gyrogroup of non-measurable cardinality is strongly $\sigma$-discrete and every submaximal strongly topological gyrogroup of non-measurable cardinality is hereditarily paracompact. More important, they extended the cardinality of topological gyrogroups since they gave an example to show that for any cardinality $\kappa >\omega$, there exists a gyrogroup $G$ with subgyrogroup $H$ of the cardinality $\kappa$ such that $H$ is not a group, see \cite[Example 3.1]{BL3}. Moreover, when Bao and Lin researched the quotient spaces of strongly topological gyrogroups, they posed the concept of an admissible subgyrogroup generated from the invariant neighborhood base, and they proved in \cite{BL1} that if $H$ is an admissible subgyrogroup generated from $\mathscr{U}$ of a strongly topological gyrogroup $G$, where $\mathscr U$ is a symmetric neighborhood base at $0$, then the left coset space $G/H$ is submetrizable. Recently, Bao, Ling and Xu \cite{BLX} were continue to study the quotient spaces of strongly topological gyrogroups and achieved many well results combining generalized metric properties, such as $\aleph_{0}$-space, cosmic space, stratifiable space and so on. Bao, Zhang and Xu \cite{BZX1} showed that $G$ is a strongly countably complete strongly topological gyrogroup, then $G$ contains a closed, countably compact admissible subgyrogroup $P$ such that the quotient space $G/P$ is metrizable and the canonical homomorphism $\pi :G\rightarrow G/P$ is closed.

In topological gyrogroups, it was proved in \cite{CZ} that the first-countability and metrizability are equivalent, so it is natural to consider whether they are equivalent in quotient spaces of topological gyrogroups. In this paper, we show that if $(G,\tau ,\oplus)$ is a strongly topological gyrogroup with a symmetric neighborhood base $\mathscr U$ at $0$ and $H$ is an admissible subgyrogroup generated from $\mathscr U$, then the first-countability and metrizability of $G/H$ are also equivalent. Moreover, if the admissible subgyrogroup $H$ is neutral, we show that $G/H$ is Fr\'echet-Urysohn with an $\omega^{\omega}$-base if and only if it is metrizable. Finally, we investigate the cardinalities of the quotient spaces and prove that if $H$ is neutral, then $\pi\chi(G/H)=\chi(G/H)$ and $\pi\omega (G/H)=\omega (G/H)$.

\smallskip
\section{Preliminaries}
In this section, we introduce the necessary notations, terminologies and some facts about topological gyrogroups.

Throughout this paper, all topological spaces are assumed to be Hausdorff, unless otherwise is explicitly stated. Let $\mathbb{N}$ be the set of all positive integers. Let $X$ be a topological space and $A \subseteq X$ be a subset of $X$.
  The {\it closure} of $A$ in $X$ is denoted by $\overline{A}$ and the
  {\it interior} of $A$ in $X$ is denoted by $\mbox{Int}(A)$. The readers may consult \cite{AA, E, linbook} for notations and terminologies not explicitly given here.
\begin{definition}\cite{AW}
Let $G$ be a nonempty set, and let $\oplus: G\times G\rightarrow G$ be a binary operation on $G$. Then the pair $(G, \oplus)$ is called a {\it groupoid}. A function $f$ from a groupoid $(G_{1}, \oplus_{1})$ to a groupoid $(G_{2}, \oplus_{2})$ is called a {\it groupoid homomorphism} if $f(x\oplus_{1}y)=f(x)\oplus_{2} f(y)$ for any elements $x, y\in G_{1}$. Furthermore, a bijective groupoid homomorphism from a groupoid $(G, \oplus)$ to itself will be called a {\it groupoid automorphism}. We write $\mbox{Aut}(G, \oplus)$ for the set of all automorphisms of a groupoid $(G, \oplus)$.
\end{definition}

\begin{definition}\cite{UA}
Let $(G, \oplus)$ be a groupoid. The system $(G,\oplus)$ is called a {\it gyrogroup}, if its binary operation satisfies the following conditions:

\smallskip
(G1) There exists a unique identity element $0\in G$ such that $0\oplus a=a=a\oplus0$ for all $a\in G$.

\smallskip
(G2) For each $x\in G$, there exists a unique inverse element $\ominus x\in G$ such that $\ominus x \oplus x=0=x\oplus (\ominus x)$.

\smallskip
(G3) For all $x, y\in G$, there exists $\mbox{gyr}[x, y]\in \mbox{Aut}(G, \oplus)$ with the property that $x\oplus (y\oplus z)=(x\oplus y)\oplus \mbox{gyr}[x, y](z)$ for all $z\in G$.

\smallskip
(G4) For any $x, y\in G$, $\mbox{gyr}[x\oplus y, y]=\mbox{gyr}[x, y]$.
\end{definition}

Notice that a group is a gyrogroup $(G,\oplus)$ such that $\mbox{gyr}[x,y]$ is the identity function for all $x, y\in G$. The definition of a subgyrogroup is as follows.

\begin{definition}\cite{ST}
Let $(G,\oplus)$ be a gyrogroup. A nonempty subset $H$ of $G$ is called a {\it subgyrogroup}, denoted
by $H\leq G$, if the following statements hold:

\smallskip
(i) The restriction $\oplus| _{H\times H}$ is a binary operation on $H$, i.e. $(H, \oplus| _{H\times H})$ is a groupoid.

\smallskip
(ii) For any $x, y\in H$, the restriction of $\mbox{gyr}[x, y]$ to $H$, $\mbox{gyr}[x, y]|_{H}$ : $H\rightarrow \mbox{gyr}[x, y](H)$, is a bijective homomorphism.

\smallskip
(iii) $(H, \oplus|_{H\times H})$ is a gyrogroup.

\smallskip
Furthermore, a subgyrogroup $H$ of $G$ is said to be an {\it $L$-subgyrogroup} \cite{ST}, denoted
by $H\leq_{L} G$, if $\mbox{gyr}[a, h](H)=H$ for all $a\in G$ and $h\in H$.

\end{definition}

\begin{definition}\cite{AW}
A triple $(G, \tau, \oplus)$ is called a {\it topological gyrogroup} if the following statements hold:

\smallskip
(1) $(G, \tau)$ is a topological space.

\smallskip
(2) $(G, \oplus)$ is a gyrogroup.

\smallskip
(3) The binary operation $\oplus: G\times G\rightarrow G$ is jointly continuous while $G\times G$ is endowed with the product topology, and the operation of taking the inverse $\ominus (\cdot): G\rightarrow G$, i.e. $x\rightarrow \ominus x$, is also continuous.
\end{definition}

\begin{definition}{\rm (\cite{BL})}\label{d11}
Let $G$ be a topological gyrogroup. We say that $G$ is a {\it strongly topological gyrogroup} if there exists a neighborhood base $\mathscr U$ of $0$ such that, for every $U\in \mathscr U$, $\mbox{gyr}[x, y](U)=U$ for any $x, y\in G$. For convenience, we say that $G$ is a strongly topological gyrogroup with neighborhood base $\mathscr U$ at $0$.
\end{definition}

It was given that the M\"{o}bius gyrogroup with the standard topology is a strongly topological gyrogroup but not a topological group.

\begin{example}\cite{AW}\label{2lz1}
The M\"{o}bius gyrogroup with the standard topology is a strongly topological gyrogroup but not a topological group.
\end{example}

Let $\mathbb{D}$ be the complex open unit disk $\{z\in \mathbb{C}:|z|<1\}$ and equipped with the standard topology. In \cite[Example 2]{AW}, define a M\"{o}bius addition $\oplus _{M}: \mathbb{D}\times \mathbb{D}\rightarrow \mathbb{D}$ to be a function such that $$a\oplus _{M}b=\frac{a+b}{1+\bar{a}b}\ \mbox{for all}\ a, b\in \mathbb{D}.$$ Then $(\mathbb{D}, \oplus _{M})$ is a gyrogroup, and it follows from \cite[Example 2]{AW} that $$gyr[a, b](c)=\frac{1+a\bar{b}}{1+\bar{a}b}c\ \mbox{for any}\ a, b, c\in D.$$ For any $n\in\omega$, let $U_{n}=\{x\in \mathbb{D}: |x|\leq \frac{1}{n}\}$. Then, $\mathscr U=\{U_{n}: n\in \omega\}$ is a neighborhood base of $0$. Moreover, we observe that $|\frac{1+a\bar{b}}{1+\bar{a}b}|=1$. Therefore, we obtain that $gyr[x, y](U)\subset U$, for any $x, y\in \mathbb{D}$ and each $U\in \mathscr U$, then it follows that $gyr[x, y](U)=U$ by \cite[Proposition 2.6]{ST}. Hence, $(\mathbb{D}, \oplus _{M})$ is a strongly topological gyrogroup. However, $(\mathbb{D}, \oplus _{M})$ is not a group \cite[Example 2]{AW}.

It was proved in \cite[Example 3.2]{BL} that there exists a strongly topological gyrogroup which has an infinite L-subgyrogroup. They constructed a strongly topological gyrogroup by the production of a feathered non-metrizable topological group with a strongly topological gyrogroup with a non-trivial $L$-subgyrogroup. Indeed, every strongly topological gyrogroup has some subgyrogroups which are invariant under the groupoid automorphisms of the strongly topological gyrogroup, so these subgyrogroups are all $L$-subgyrogroups.

\begin{theorem}
Every strongly topological gyrogroup $G$ contains some subgyrogroups which are invariant under the groupoid automorphisms of $G$.
\end{theorem}

\begin{proof}
Let $G$ be a strongly topological gyrogroup with a symmetric neighborhood base $\mathscr U$ at $0$. For arbitrary $U\in \mathscr{U}$, we know that $U=\ominus U$ and $0\in U$. Put $U_{0}=U$ and $U_{n}=\ominus (U_{n-1}\oplus U_{n-1})\cup (U_{n-1}\oplus U_{n-1})$ for any $n\in \mathbb{N}$. Set $H=\bigcup_{n\in \mathbb{N}}U_{n}$.

{\bf Claim 1}. $H$ is a subgyrogroup of $G$.

By the construction of $H$, it is clear that $H=\ominus H$. For arbitrrary $m,n\in H$, we know that $m,n\in U_{n}$, for some $n\in \mathbb{N}$. Then, $m\oplus n\in U_{n}\oplus U_{n}\subset U_{n+1}\subset H$, so $H$ is a subgyrogroup by \cite[Proposition 4.1]{ST}.

{\bf Claim 2}. $H$ is invariant under the groupoid automorphisms of $G$.

For any $x,y\in G$, we show that $gyr[x,y](U_{n})=U_{n}$ for any $n\in \mathbb{N}$. If $n=0$, it is obvious that $gyr[x,y](U)=U$ since $U\in \mathscr{U}$. Assume that if $n=k$, we have $gyr[x,y](U_{k})=U_{k}$. Then when $n=k+1$,
\begin{eqnarray}
gyr[x,y](U_{k+1})&=&gyr[x,y](\ominus (U_{k}\oplus U_{k})\cup (U_{k}\oplus U_{k}))\nonumber\\
&=&\ominus (gyr[x,y](U_{k})\oplus gyr[x,y](U_{k}))\cup (gyr[x,y](U_{k})\oplus gyr[x,y](U_{k}))\nonumber\\
&=&\ominus (U_{k}\oplus U_{k})\cup (U_{k}\oplus U_{k})\nonumber\\
&=&U_{k+1}\nonumber.
\end{eqnarray}

Therefore, we obtain that $gyr[x,y](H)=gyr[x,y](\bigcup_{n\in \mathbb{N}}U_{n})=\bigcup gyr[x,y](U_{n})=\bigcup_{n\in \mathbb{N}}U_{n}=H$ for any $x,y\in G$.
\end{proof}

Since the subgyrogroup $H$ is invariant under the groupoid automorphisms of $G$, it is natural that $H$ is an $L$-subgyrogroup of $G$. Then we recall the following concept of the coset space of a topological gyrogroup.

Let $(G, \tau, \oplus)$ be a topological gyrogroup and $H$ an $L$-subgyrogroup of $G$. It follows from \cite[Theorem 20]{ST} that $G/H=\{a\oplus H:a\in G\}$ is a partition of $G$. We denote by $\pi$ the mapping $a\mapsto a\oplus H$ from $G$ onto $G/H$. Clearly, for each $a\in G$, we have $\pi^{-1}\{\pi(a)\}=a\oplus H$.
Denote by $\tau (G)$ the topology of $G$. In the set $G/H$, we define a family $\tau (G/H)$ of subsets as follows: $$\tau (G/H)=\{O\subset G/H: \pi^{-1}(O)\in \tau (G)\}.$$

The following concept of an admissible subgyrogroup of a strongly topological gyrogroup was first introduced in \cite{BL1}, which plays an important role in this paper.

\begin{definition}\cite{BL1}
A subgyrogroup $H$ of a topological gyrogroup $G$ is called {\it admissible} if there exists a sequence $\{U_{n}:n\in \omega\}$ of open symmetric neighborhoods of the identity $0$ in $G$ such that $U_{n+1}\oplus (U_{n+1}\oplus U_{n+1})\subset U_{n}$ for each $n\in \omega$ and $H=\bigcap _{n\in \omega}U_{n}$.
\end{definition}

If $G$ is a strongly topological gyrogroup with a symmetric neighborhood base $\mathscr U$ at $0$ and each $U_{n}\in \mathscr U$, we say that the admissible topological subgyrogroup is generated from $\mathscr U$. It was shown in \cite{BL2} that if $G$ is a strongly topological gyrogroup with a symmetric neighborhood base $\mathscr U$ at $0$, then each admissible topological subgyrogroup $H$ generated from $\mathscr U$ is a closed $L$-subgyrogroup of $G$.

\section{Coset with admissible subgyrogroups}

In this section, we show that if $(G,\tau ,\oplus)$ is a strongly topological gyrogroup with a symmetric neighborhood base $\mathscr U$ at $0$ and $H$ is an admissible subgyrogroup generated from $\mathscr U$, then the first-countability and metrizability are equivalent in the quotient space $G/H$.

Moreover, it was proved in \cite{BZX2} that a topological gyrogroup is metrizable if and only if it is Fr\'echet-Urysohn with an $\omega^{\omega}$-base, then it is natural to consider whether it holds in the quotient spaces of topological gyrogroups. Here we just investigate the quotient spaces of strongly topological gyrogroups. Therefore, we pose the following problem.

\begin{question}
Let $(G,\tau ,\oplus)$ be a strongly topological gyrogroup with a symmetric neighborhood base $\mathscr U$ at $0$ and $H$ an admissible subgyrogroup generated from $\mathscr U$, whether the condition $G/H$ is metrizable if and only if it is Fr\'echet-Urysohn with an $\omega^{\omega}$-base also holds? In particular, what if $H$ is neutral?
\end{question}

Then we give an affirmative answer about the question when $H$ is neutral, see Corollary \ref{3tl1}. First, we recall some lemmas.

\begin{lemma}\cite{BL1}\label{yl1}
Let $(G, \tau, \oplus)$ be a strongly topological gyrogroup with a symmetric neighborhood base $\mathscr U$ at $0$. If $H$ is an admissible subgyrogroup generated from $\mathscr U$, then the coset space $G/H$ is a homogenous space.
\end{lemma}

\begin{lemma}\label{s}\cite{BL}
Let $G$ be a strongly topological gyrogroup with the symmetric neighborhood base $\mathscr{U}$ at $0$, and let $\{U_{n}: n\in\mathbb{N}\}$ and $\{V(m/2^{n}): n, m\in\mathbb{N}\}$ be two sequences of open neighborhoods satisfying the following conditions (1)-(5):

\smallskip
(1) $U_{n}\in\mathscr{U}$ for each $n\in \mathbb{N}$.

\smallskip
(2) $U_{n+1}\oplus U_{n+1}\subset U_{n}$, for each $n\in \mathbb{N}$.

\smallskip
(3) $V(1)=U_{0}$;

\smallskip
(4) For any $n\geq 1$, put $$V(1/2^{n})=U_{n}, V(2m/2^{n})=V(m/2^{n-1})$$ for $m=1,...,2^{n-1}$, and $$V((2m+1)/2^{n})=U_{n}\oplus V(m/2^{n-1})=V(1/2^{n})\oplus V(m/2^{n-1})$$ for each $m=1,...,2^{n-1}-1$;

\smallskip
(5) $V(m/2^{n})=G$ when $m>2^{n}$;

\smallskip
Then there exists a prenorm $N$ on $G$ that satisfies the following conditions:

\smallskip
(a) for any fixed $x, y\in G$, we have $N(\mbox{gyr}[x,y](z))=N(z)$ for any $z\in G$;

\smallskip
(b) for any $n\in \mathbb{N}$, $$\{x\in G: N(x)<1/2^{n}\}\subset U_{n}\subset\{x\in G: N(x)\leq 2/2^{n}\}.$$
\end{lemma}

\begin{theorem}\label{h}
Suppose that $(G,\tau ,\oplus)$ is a strongly topological gyrogroup with a symmetric neighborhood base $\mathscr U$ at $0$ and $H$ is an admissible subgyrogroup generated from $\mathscr U$, if the quotient space $G/H$ is first-countable, then it is metrizable.
\end{theorem}

\begin{proof}
Let $\{U_{n}:n\in \mathbb{N}\}$ be a sequence of symmetric open neighborhoods of the identity $0$ in $G$ satisfying $U_{n}\in \mathscr U$ and $U_{n+1}\oplus (U_{n+1}\oplus U_{n+1})\subset U_{n}$, for each $n\in \mathbb{N}$, and such that $H=\bigcap _{n\in \mathbb{N}}U_{n}$. By Lemma \ref{s}, there exists a continuous prenorm $N$ on $G$ which satisfies $$N(\mbox{gyr}[x,y](z))=N(z)$$ for any $x, y, z\in G$ and $$\{x\in G: N(x)<1/2^{n}\}\subset U_{n}\subset\{x\in G: N(x)\leq 2/2^{n}\},$$ for any $n\in \mathbb{N}$.

First, we show that $N(x)=0$ if and only if $x\in H$. If $N(x)=0$, then $$x\in \bigcap _{n\in \mathbb{N}}\{x\in G:N(x)<1/2^{n}\}\subset \bigcap _{n\in \mathbb{N}}U_{n}=H.$$ On the other hand, if $x\in H$, since $$H=\bigcap _{n\in \mathbb{N}}U_{n}\subset \bigcap _{n\in \mathbb{N}}\{x\in G:N(x)\leq 2/2^{n}\},$$ we have that $N(x)=0$.

We claim that $N(x\oplus h)=N(x)$ for every $x\in G$ and $h\in H$. Indeed, for every $x\in G$ and $h\in H$, $N(x\oplus h)\leq N(x)+N(h)=N(x)+0=N(x)$. Moreover, by the definition of $N$, we observe that $N(\mbox{gyr}[x,y](z))=N(z)$ for every $x,y,z\in G$. Since $H$ is a $L$-subgyrogroup, it follows that
\begin{eqnarray}
N(x)&=&N((x\oplus h)\oplus \mbox{gyr}[x,h](\ominus h))\nonumber\\
&\leq&N(x\oplus h)+N(\mbox{gyr}[x,h](\ominus h))\nonumber\\
&=&N(x\oplus h)+N(\ominus h)\nonumber\\
&=&N(x\oplus h).\nonumber
\end{eqnarray}
Therefore, $N(x\oplus h)=N(x)$ for every $x\in G$ and $h\in H$.

Now define a function $d$ from $G\times G$ to $\mathbb{R}$ by $d(x,y)=|N(x)-N(y)|$ for all $x,y\in G$. Obviously, $d$ is continuous. We show that $d$ is a pseudometric.

\smallskip
(1) For any $x, y\in G$, if $x=y$, then $d(x, y)=|N(x)-N(y)|=0$.

\smallskip
(2) For any $x, y\in G$, $d(y, x)=|N(y)-N(x)|=|N(x)-N(y)|=d(x, y)$.

\smallskip
(3) For any $x, y, z\in G$, we have
\begin{eqnarray}
d(x, y)&=&|N(x)-N(y)|\nonumber\\
&=&|N(x)-N(z)+N(z)-N(y)|\nonumber\\
&\leq&|N(x)-N(z)|+|N(z)-N(y)|\nonumber\\
&=&d(x, z)+d(z, y).\nonumber
\end{eqnarray}

If $x'\in x\oplus H$ and $y'\in y\oplus H$, then there exist $h_{1},h_{2}\in H$ such that $x'=x\oplus h_{1}$ and $y'=y\oplus h_{2}$, then $$d(x', y')=|N(x\oplus h_{1})-N(y\oplus h_{2})|=|N(x)-N(y)|=d(x, y).$$ This enables us to define a function $\varrho $ on $G/H\times G/H$ by $$\varrho (\pi _{H}(x),\pi _{H}(y))=d(\ominus x\oplus y, 0)+d(\ominus y\oplus x, 0)$$ for any $x, y\in G$.

It is obvious that $\varrho $ is continuous, and we verify that $\varrho $ is a metric on $G/H$.

\smallskip
(1) Obviously, for any $x, y\in G$, then
\begin{eqnarray}
\varrho (\pi _{H}(x),\pi _{H}(y))=0&\Leftrightarrow&d(\ominus x\oplus y, 0)=d(\ominus y\oplus x, 0)=0\nonumber\\
&\Leftrightarrow&N(\ominus x\oplus y)=N(\ominus y\oplus x)=0\nonumber\\
&\Leftrightarrow&\ominus x\oplus y\in H\ \mbox{and}\ \ominus y\oplus x\in H\nonumber\\
&\Leftrightarrow&y\in x\oplus H\ \mbox{and}\ x\in y\oplus H\nonumber\\
&\Leftrightarrow&\pi _{H}(x)=\pi _{H}(y).\nonumber
\end{eqnarray}

\smallskip
(2) For every $x,y\in G$, it is obvious that $\varrho (\pi _{H}(y), \pi _{H}(x))=\varrho (\pi _{H}(x),\pi _{H}(y))$.

\smallskip
(3) For every $x, y, z\in G$, it follows from \cite[Theorem 2.11]{UA2005} that
\begin{eqnarray}
\varrho (\pi _{H}(x),\pi _{H}(y))&=&N(\ominus x\oplus y)+N(\ominus y\oplus x)\nonumber\\
&=&N((\ominus x\oplus z)\oplus \mbox{gyr}[\ominus x,z](\ominus z\oplus y))\nonumber\\
&&+N((\ominus y\oplus z)\oplus \mbox{gyr}[\ominus y,z](\ominus z\oplus x))\nonumber\\
&\leq&N(\ominus x\oplus z)+N(\mbox{gyr}[\ominus x,z](\ominus z\oplus y))\nonumber\\
&&+N(\ominus y\oplus z)+N(\mbox{gyr}[\ominus y,z](\ominus z\oplus x))\nonumber\\
&=&N(\ominus x\oplus z)+N(\ominus z\oplus y)+N(\ominus y\oplus z)+N(\ominus z\oplus x)\nonumber\\
&=&d(\ominus x\oplus z, 0)+d(\ominus z\oplus x, 0)+d(\ominus z\oplus y, 0)+d(\ominus y\oplus z, 0)\nonumber\\
&=&\varrho (\pi _{H}(x),\pi _{H}(z))+\varrho (\pi _{H}(z),\pi _{H}(y)).\nonumber
\end{eqnarray}

Then we verify that $\varrho$ generates the quotient topology of the space $G/H$. Given any points $x\in G$, $y\in G/H$ and any $\varepsilon>0$, we define open balls, $$B(x, \varepsilon)=\{x'\in G: d(x',x)<\varepsilon\}$$ and $$B^{*}(y, \varepsilon)=\{y'\in G/H: \varrho (y',y)<\varepsilon\}$$ in $G$ and $G/H$, respectively. Obviously, if $x\in G$ and $y=\pi _{H}(x)$, then we have $B(x, \varepsilon)=\pi ^{-1}_{H}(B^{*}(y, \varepsilon))$. Therefore, the topology generated by $\varrho$ on $G/H$ is coarser than the quotient topology.

Moreover, it follows from Lemma \ref{yl1} that the quotient space $G/H$ is homogenous. Since $G/H$ is first-countable, there exists a countable base $\mathcal{V}=\{V_{n}:n\in \mathbb{N}\}$ at $\pi (0)$ in $G/H$. As $\pi$ is an open and continuous mapping from $G$ onto $G/H$ by \cite[Theorem 3.7]{BL}, $\pi(W)$ is an open neighborhood of $\pi(0)$ for arbitrary open neighborhood $W$ of $\pi^{-1}(\pi(0))$. We can find $V_{m}\in \mathcal{V}$ such that $\pi(0)\subset V_{m}\subset \pi(W)$. Then $\pi^{-1}(\pi(0))\subset \pi^{-1}(V_{m})\subset W$. Then $H$ has a countable character in $G$. Suppose that the preimage $O=\pi ^{-1}_{H}(Q)$ is open in $G$, where $Q$ is a non-empty subset of $G/H$. For every $y\in Q$, we have $\pi ^{-1}_{H}(y)=x\oplus H\subset O$, where $x$ is an arbitrary point of the fiber $\pi ^{-1}_{H}(y)$. Since $\{\pi^{-1}(V_{n}): n\in \omega\}$ is a base for $G$ at $H$, there exists $n\in \omega$ such that $x\oplus \pi^{-1}(V_{n})\subset O$. Then there exists $\delta>0$ such that $B(x, \delta)\subset x\oplus \pi^{-1}(V_{n})$. Therefore, we have $\pi ^{-1}_{H}(B^{*}(y, \delta))=B(x, \delta)\subset x\oplus \pi^{-1}(V_{n})\subset O$. It follows that $B^{*}(y, \delta)\subset Q$. So the set $Q$ is the union of a family of open balls in $(G/H, \varrho)$. Hence, $Q$ is open in $(G/H, \varrho)$, which proves that the metric and quotient topologies on $G/H$ coincide.
\end{proof}

\begin{corollary}
Let $(G,\tau ,\oplus)$ be a strongly topological gyrogroup with a symmetric neighborhood base $\mathscr U$ at $0$ and $H$ an admissible subgyrogroup generated from $\mathscr U$. Then $G/H$ is metrizable if and only if it is first-countable.
\end{corollary}

\begin{corollary}
Let $(G,\tau ,\oplus)$ be a strongly topological gyrogroup with a symmetric neighborhood base $\mathscr U$ at $0$ and $H$ an admissible subgyrogroup generated from $\mathscr U$. If $H$ has countable character in $G$, then the quotient space $G/H$ is metrizable.
\end{corollary}

\begin{proof}
Since $H$ has countable character in $G$, we can fix a family $\{U_{n}:n\in \mathbb{N}\}$ of open sets such that $H\subset \bigcap_{n\in \mathbb{N}} U_{n}$ and for every open subset $U$ with $H\subset U$, there exists $n\in \mathbb{N}$ such that $H\subset U_{n}\subset U$. Let $\pi$ be the natural homomorphism from $G$ to $G/H$. We claim that $\{ \pi(U_{n}):n\in \mathbb{N}\}$ is a neighborhood base at $\pi(0)$ in $G/H$. Indeed, for every open neighborhood $V$ of $\pi(0)$ in $G/H$, we have that $H\subset \pi^{-1}(V)$. Obviously, $\pi^{-1}(V)$ is open in $G$. Hence, there exists $n\in \mathbb{N}$ such that $H\subset U_{n}\subset \pi^{-1}(V)$. Thus $\pi (U_{n})\subset V$. Consequently,$\{\pi (U_{n}):n\in \mathbb{N}\}$ is a neighborhood base at $\pi(0)$ in $G/H$. Moreover, it follows from Lemma \ref{yl1} that the quotient space $G/H$ is homogenous. Therefore, $G/H$ is first-countable and hence metrizable by Theorem \ref{h}.
\end{proof}

An $L$-subgyrogroup $H$ of a topological gyrogroup $G$ is called {\it neutral} if for every open neighborhood $U$ of $0$ in $G$, there exists an open neighborhood $V$ of $0$ such that $H\oplus V\subset U\oplus H$ (equivalently, $V\oplus H\subset H\oplus U$).

\begin{lemma}\label{4yl1}
Let $(G,\tau ,\oplus)$ be a strongly topological gyrogroup with a symmetric neighborhood base $\mathscr U$ at $0$ and $H$ an admissible subgyrogroup generated from $\mathscr U$. If $H$ is neutral and the sequences $\{a_{n}\}_{n\in \mathbb{N}}$ and $\{b_{n}\}_{n\in \mathbb{N}}$ satisfy $\{\pi (a_{n})\}_{n\in \mathbb{N}}$ and $\{\pi (b_{n})\}_{n\in \mathbb{N}}$ both converging to $\pi (0)$, then $\{\pi (\ominus a_{n})\}_{n\in \mathbb{N}}$ and $\{\pi (a_{n}\oplus b_{n})\}_{n\in \mathbb{N}}$ both converge to $\pi (0)$.
\end{lemma}

\begin{proof}
By the hypothesis, $G$ is a strongly topological gyrogroup with a symmetric neighborhood base $\mathscr U$ at $0$ and $H$ is an admissible subgyrogroup generated from $\mathscr U$, so $gyr[x,y](H)=H$ for each $x,y\in G$ by the proof of \cite[Theorem 3.2]{BL2}. Let $O$ be an open neighborhood of $\pi (0)$ in $G/H$. There exists $U\in \mathscr U$ such that $\pi (U\oplus U)\subset O$. Since $H$ is neutral, we can find $V\in \mathscr U$ with $H\oplus V\subset U\oplus H$.

First, we show that $\{\pi (\ominus a_{n})\}_{n\in \mathbb{N}}$ converges to $\pi (0)$.

Since $\{\pi (a_{n})\}_{n\in \mathbb{N}}$ converges to $\pi (0)$, there exists $m\in \mathbb{N}$ such that $\pi (a_{n})\in \pi (V)$ for each $n\geq m$, that is, $a_{n}\in V\oplus H$ for each $n\leq m$, then there exists $v_{n}\in V$ and $h_{n}\in H$ such that $a_{n}=v_{n}\oplus h_{n}$. Then $$\ominus a_{n}=\ominus (v_{n}\oplus h_{n})=gyr[v_{n},h_{n}](\ominus h_{n}\ominus v_{n})\in gyr[v_{n},h_{n}](H\oplus V)=H\oplus V\subset U\oplus H.$$ Therefore, $\pi (\ominus a_{n})\in \pi (U\oplus H)=\pi (U)\subset \pi (U\oplus U)\subset O$ for each $n\geq m$. Thus, $\{\pi (\ominus a_{n})\}_{n\in \mathbb{N}}$ converges to $\pi (0)$.

Then we show that $\{\pi (a_{n}\oplus b_{n})\}_{n\in \mathbb{N}}$ converges to $\pi (0)$.

Since $\mathscr U$ is a symmetric neighborhood base at $0$ in $G$, there exists $W\in \mathscr U$ with $W\subset V\cap U$. As $\{\pi (a_{n})\}_{n\in \mathbb{N}}$ and $\{\pi (b_{n})\}_{n\in \mathbb{N}}$ both converge to $\pi (0)$, we can find $m\in \mathbb{N}$ such that $\pi (a_{n})\in \pi (W)$ and $\pi (b_{m})\in \pi (W)$ for each $n\geq m$, that is, $a_{n}\in W\oplus H$ and $b_{n}\in W\oplus H$. Then there are $w_{n_{1}},w_{n_{2}}\in W$ and $h_{n_{1}},h_{n_{2}}\in H$ such that $a_{n}=w_{n_{1}}\oplus h_{n_{1}}$ and $b_{n}=w_{n_{2}}\oplus h_{n_{2}}$. Then

\begin{eqnarray}
a_{n}\oplus b_{n}&=&(w_{n_{1}}\oplus h_{n_{1}})\oplus (w_{n_{2}}\oplus h_{n_{2}})\nonumber\\
&=&((w_{n_{1}}\oplus h_{n_{1}})\oplus w_{n_{2}})\oplus gyr[w_{n_{1}}\oplus h_{n_{1}},w_{n_{2}}](h_{n_{2}})\nonumber\\
&=&(w_{n_{1}}\oplus (h_{n_{1}}\oplus gyr[h_{n_{1}},w_{n_{1}}](w_{n_{2}})))\oplus gyr[w_{n_{1}}\oplus h_{n_{1}},w_{2}](h_{2})\nonumber\\
&\in&(W\oplus (H\oplus W))\oplus H\nonumber\\
&\subset&(U\oplus (U\oplus H))\oplus H\nonumber\\
&\subset&((U\oplus U)\oplus gyr[U,U](H))\oplus H\nonumber\\
&=&((U\oplus U)\oplus H)\oplus H.\nonumber
\end{eqnarray}

Therefore, we obtain that $\pi (a_{n}\oplus b_{n})\in \pi (((U\oplus U)\oplus H)\oplus H)=\pi (U\oplus U)\subset O$ for each $n\in \mathbb{N}$. Hence, $\{\pi (a_{n}\oplus b_{n})\}_{n\in \mathbb{N}}$ converges to $\pi (0)$.
\end{proof}

\begin{definition}\cite{FS}
Let $X$ be a topological space. A space is called Fr\'echet-Urysohn at a point $x\in X$ if for every $A\subset X$ with $x\in \overline{A}\subset X$ there is a sequence $\{x_{n}\}_{n\in \mathbb{N}}$ in $A$ such that $\{x_{n}\}_{n\in \mathbb{N}}$ converges to $x$ in $X$. A space is called Fr\'echet-Urysohn if it is Fr\'echet-Urysohn at every point $x\in X$.
\end{definition}

\begin{definition}\cite{CZ1}
A topological space $X$ is called a {\it strong $\alpha_{4}$-space} if for any subset $\{x_{m,n}:m,n\in \mathbb{N}\}\subset X$ with $lim_{n\rightarrow \infty}x_{m,n}=x\in X$ for each $m\in \mathbb{N}$, there are strictly increasing sequences of natural numbers $\{i_{k}\}_{k\in \mathbb{N}}$ and $\{j_{k}\}_{k\in \mathbb{N}}$ such that $lim _{k\rightarrow \infty}x_{i_{k},j_{k}}=x$.
\end{definition}

\begin{theorem}\label{4dl1}
Let $(G,\tau ,\oplus)$ be a strongly topological gyrogroup with a symmetric neighborhood base $\mathscr U$ at $0$ and $H$ an admissible subgyrogroup generated from $\mathscr U$. If $H$ is neutral and $G/H$ is Fr\'echet-Urysohn, then $G/H$ is a strong $\alpha_{4}$-space.
\end{theorem}

\begin{proof}
Let $\{x_{m,n}:m,n\in \mathbb{N}\}\subset G/H$ be such that $lim_{n\rightarrow \infty}x_{m,n}=\pi (0)$ for each $m\in \mathbb{N}$. We show that there exist strictly increasing sequences of natural numbers $\{i_{k}\}_{k\in \mathbb{N}}$ and $\{j_{k}\}_{k\in \mathbb{N}}$ such that $lim_{k\rightarrow \infty}x_{i_{k},j_{k}}=\pi (0)$.

If $G/H$ is discrete, it is trivial. Then we assume that $G/H$ is non-discrete. We can find a sequence $\{a_{m}\}_{m\in \mathbb{N}}\subset G/H$ with $lim_{m\rightarrow \infty}a_{m}=\pi (0)$ and $a_{m}\not=\pi (0)$ for each $m\in \mathbb{N}$. For each $a\in G/H$, take a point $g_{a}\in G$ such that $\pi (g_{a})=a$. Put
$$y_{m,l}=\left\{
\begin{matrix}
\pi (g_{a_{m}}\oplus g_{x_{m},l+m})&, &\pi (g_{a_{m}}\oplus g_{x_{m},l+m})\not=\pi (0);\\
a_{m}&, &otherwise.
\end{matrix}
\right.$$

Set $M=\{y_{m,l}:(m,l)\in \mathbb{N}\times \mathbb{N}\}$. It is clear that $\pi (0)\not\in M$. We show that $\pi (0)\in \overline{M}$.

Indeed, for each open neighborhood $U$ of $\pi (0)$ in $G/H$, since $lim_{m\rightarrow \infty}a_{m}=\pi (0)$, there exists $m\in \mathbb{N}$ such that $a_{m}\in U$. By $lim_{n\rightarrow \infty}x_{m,n}=\pi (0)$, we have $lim_{l\rightarrow \infty}y_{m,l}=a_{m}$. Thus, we can find $l\in \mathbb{N}$ with $y_{m,l}\in U$. We obtain that $\pi (0)\in \overline{M}$. By the hypothesis, $G/H$ is Fr\'echet-Urysohn, so there exists a sequence $\{(m_{k},l_{k})\}_{k\in \mathbb{N}}$ with $lim_{k\rightarrow \infty}y_{m_{k},l_{k}}=\pi (0)$.

{\bf Case 1.} The sequence $\{l_{k}\}_{k\in \mathbb{N}}$ is bounded.

Choose a subsequence if necessary, then assume that $l_{k}=r, k=1,2,\cdot \cdot \cdot $, for some natural number $r$. Since $lim_{k\rightarrow \infty}y_{m_{k},r}=lim_{k\rightarrow \infty}y_{m_{k},l_{k}}=\pi (0)$ and $y_{m_{k},r}\not=\pi (0)$, we know that $lim_{k\rightarrow \infty}m_{k}=\infty$. Choosing once more a subsequence, we assume further that $m_{k}< m_{k+1}$ for each $k\in \mathbb{N}$.

{\bf Subcase 1.1.} The set $N_{1}=\{k\in\mathbb{N}:y_{m_{k},r}=a_{m_{k}}\}$ is infinite.

Denote $N_{1}=\{p_{1},p_{2},\cdot \cdot\cdot\}$ with $p_{k}< p_{k+1}$ for each $k\in \mathbb{N}$. Then $\pi (g_{a_{m_{p_{k}}}}\oplus g_{x_{m_{p_{k}}},x+m_{p_{k}}})=\pi (0)$ for each $k\in \mathbb{N}$. As $lim_{k\rightarrow\infty}\pi (g_{a_{m_{p_{k}}}})=lim_{k\rightarrow\infty}a_{m_{p_{k}}}=\pi (0)$, we have that $lim_{k\rightarrow\infty}\pi [\ominus g_{a_{m_{p_{k}}}}\oplus (g_{a_{m_{p_{k}}}}\oplus g_{x_{m_{p_{k}}},r+m_{p_{k}}})]=\pi (0)$ by Lemma \ref{4yl1}, that is, $lim_{k\rightarrow\infty}x_{m_{p_{k}},r+m_{p_{k}}}=\pi (0)$. For each $k\in \mathbb{N}$, put $i_{k}=m_{p_{k}}$ and $j_{k}=r+m_{p_{k}}$, then $\{i_{k}\}_{k\in \mathbb{N}}$ and $\{j_{k}\}_{k\in \mathbb{N}}$ are strictly increasing sequences such that $lim_{k\rightarrow\infty}x_{i_{k},j_{k}}=\pi (0)$.

{\bf Subcase 1.2.} The set $N_{1}=\{k\in\mathbb{N}:y_{m_{k},r}=a_{m_{k}}\}$ is finite.

Then put $N_{2}=\{k\in \mathbb{N}:y_{m_{k},r}\not=a_{m_{k}}\}$ is infinite. Denote $N_{2}=\{q_{1},q_{2},\cdot \cdot\cdot\}$ with $q_{k}< q_{k+1}$ for each $k\in \mathbb{N}$. Then $y_{m_{q_{k}},l_{q_{k}}}=\pi (g_{a_{m_{q_{k}}}}\oplus g_{x_{m_{q_{k}}},x+m_{q_{k}}})$ for each $k\in \mathbb{N}$. Since $lim_{k\rightarrow\infty}\pi (g_{a_{m_{q_{k}}}})=lim_{k\rightarrow\infty}a_{m_{q_{k}}}=\pi (0)$, we have that $lim_{k\rightarrow\infty}\pi [\ominus g_{a_{m_{q_{k}}}}\oplus (g_{a_{m_{q_{k}}}}\oplus g_{x_{m_{q_{k}}},r+m_{q_{k}}})]=\pi (0)$ by Lemma \ref{4yl1}, that is, $lim_{k\rightarrow\infty}x_{m_{q_{k}},r+m_{q_{k}}}=\pi (0)$. For each $k\in \mathbb{N}$, put $i_{k}=m_{q_{k}}$ and $j_{k}=r+m_{q_{k}}$, then $\{i_{k}\}_{k\in \mathbb{N}}$ and $\{j_{k}\}_{k\in \mathbb{N}}$ are also strictly increasing sequences such that $lim_{k\rightarrow\infty}x_{i_{k},j_{k}}=\pi (0)$.

{\bf Case 2.} The sequence $\{l_{k}\}_{k\in \mathbb{N}}$ is not bounded.

Assume that $\{l_{k}\}_{k\in \mathbb{N}}$ is strictly increasing, then $lim_{k\rightarrow \infty}m_{k}=\infty$. Otherwise, taking a subsequence if necessary, there exists $s\in \mathbb{N}$ such that $m_{k}=s$ for each $k\in \mathbb{N}$. Since $\{l_{k}\}_{k\in \mathbb{N}}$ is strictly increasing, $lim_{k\rightarrow\infty}x_{s,s+l_{k}}=\pi (0)$. By $lim_{k\rightarrow\infty}y_{s,l_{k}}=\pi (0)$, $a_{m_{k}}=a_{s}=\pi (0)$, which is a contradiction with the choice of $\{a_{m}\}_{m\in \mathbb{N}}$. Thus $lim_{k\rightarrow\infty}m_{k}=\infty$. Then there exists a strictly increasing sequence $\{n_{k}\}_{k\in \mathbb{N}}$ of $\mathbb{N}$ such that $m_{n_{1}}< m_{n_{2}}<\cdot\cdot\cdot$. As $$lim_{k\rightarrow\infty}\pi (g_{a_{m_{n_{k}}}}\oplus g_{x_{m_{n_{k}}},l_{n_{k}}+m_{n_{k}}})=lim_{k\rightarrow\infty}y_{m_{n_{k}},l_{n_{k}}}=\pi (0)$$ and $$lim_{k\rightarrow\infty}\pi (g_{a_{m_{n_{k}}}})=lim_{k\rightarrow\infty}a_{m_{n_{k}}}=\pi (0),$$ by Lemma \ref{4yl1}, we have that $$lim_{k\rightarrow\infty}\pi [\ominus g_{a_{m_{n_{k}}}}\oplus (g_{a_{m_{n_{k}}}}\oplus g_{x_{m_{n_{k}}},l_{n_{k}}+m_{n_{k}}})]=\pi (0),$$ that is, $lim_{k\rightarrow\infty}x_{m_{n_{k}},l_{n_{k}}+m_{n_{k}}}=\pi (0)$. For each $k\in \mathbb{N}$, put $i_{k}=m_{n_{k}}$ and $j_{k}=l_{n_{k}}+m_{n_{k}}$, so $\{i_{k}\}_{k\in \mathbb{N}}$ and $\{j_{k}\}_{k\in \mathbb{N}}$ are strictly increasing sequences such that $lim_{k\rightarrow\infty}x_{i_{k},j_{k}}=\pi (0)$.
\end{proof}

\begin{definition}\cite{BT,LPT,GK}
A point $x$ of a topological space $X$ is said to have a {\it neighborhood $\omega^{\omega}$-base} or a {\it local $\mathfrak{G}$-base} if there exists a base of neighborhoods at $x$ of the form $\{U_{\alpha}(x):\alpha \in \mathbb{N}^{\mathbb{N}}\}$ such that $U_{\beta}(x)\subset U_{\alpha}(x)$ for all elements $\alpha \leq \beta$ in $\mathbb{N}^{\mathbb{N}}$, where $\mathbb{N}^{\mathbb{N}}$ consisting of all functions from $\mathbb{N}$ to $\mathbb{N}$ is endowed with the natural partial order, ie., $f\leq g$ if and only if $f(n)\leq g(n)$ for all $n\in \mathbb{N}$. The space $X$ is said to have an {\it $\omega^{\omega}$-base} or a {\it $\mathfrak{G}$-base} if it has a neighborhood $\omega^{\omega}$-base or a local $\mathfrak{G}$-base at every point $x\in X$.
\end{definition}

Suppose that $(G,\tau ,\oplus)$ is a strongly topological gyrogroup with a symmetric neighborhood base $\mathscr U$ at $0$ and $H$ is an admissible subgyrogroup generated from $\mathscr U$. Suppose further that the quotient space $G/H$ has an $\omega^{\omega}$-base $\{U_{\alpha}:\alpha \in \mathbb{N}^{\mathbb{N}}\}$. Set $$I_{k}(\alpha)=\{\beta \in \mathbb{N}^{\mathbb{N}}:\beta _{i}=\alpha _{i} ~for~i=1,...,k\}, ~~and~~D_{k}(\alpha)=\bigcap_{\beta \in I_{k}(\alpha)}U_{\beta},$$ where $\alpha =(\alpha _{i})_{i\in \mathbb{N}}\in \mathbb{N}^{\mathbb{N}}$ and $k\in \mathbb{N}$. It is clear that $\{D_{k}(\alpha)\}_{k\in \mathbb{N}}$ is an increasing sequence of subsets of $G$ and contains $\pi (0)$.

\begin{lemma}\cite{GKL}\label{GKL1}
Let $\alpha =(\alpha _{i})_{i\in \mathbb{N}}\in \mathbb{N}^{\mathbb{N}}$ and $\beta _{k}=(\beta ^{k}_{i})_{i\in \mathbb{N}}\in I_{k}(\alpha)$ for every $k\in \mathbb{N}$. Then there is $\gamma \in \mathbb{N}^{\mathbb{N}}$ such that $\alpha \leq \gamma$ and $\beta _{k}\leq \gamma$ for every $k\in \mathbb{N}$.
\end{lemma}

\begin{theorem}\label{fdl1}
Let $(G,\tau ,\oplus)$ be a strongly topological gyrogroup with a symmetric neighborhood base $\mathscr U$ at $0$ and $H$ an admissible subgyrogroup generated from $\mathscr U$. If $H$ is neutral and $G/H$ is Fr\'echet-Urysohn with an $\omega^{\omega}$-base $\{U_{\alpha}:\alpha \in \mathbb{N}^{\mathbb{N}}\}$, then $G/H$ is first-countable.
\end{theorem}

\begin{proof}
First, we establish the following:

{\bf Claim.} There is $k\in \mathbb{N}$ such that $D_{k}(\alpha)$ is a neighborhood of $\pi (0)$ for each $\alpha \in \mathbb{N}^{\mathbb{N}}$.

Suppose not, that is, we can find $\alpha \in \mathbb{N}^{\mathbb{N}}$ such that $D_{k}(\alpha)$ is not a neighborhood of $\pi (0)$, for any $k\in \mathbb{N}$. It means that $\pi (0)\in \overline{G/H\setminus D_{k}(\alpha)}$ for any $k\in \mathbb{N}$. By the hypothesis, $G/H$ is Fr\'echet-Urysohn. Therefore, there exists a sequence $\{x_{n,k}\}_{n\in \mathbb{N}}$ in $G/H\setminus D_{k}(\alpha)$ which converges to $\pi (0)$. It follows from Theorem \ref{4dl1} that $G/H$ is a strong $\alpha _{4}$-space, so there are strictly increasing sequences $(n_{i})_{i\in \mathbb{N}}$ and $(k_{i})_{i\in \mathbb{N}}$ of natural numbers such that $lim_{i\rightarrow \infty}x_{n_{i},k_{i}}=\pi (0)$. For every $i\in \mathbb{N}$, there exists $\beta _{k_{i}}\in I_{k_{i}}(\alpha)$ such that $x_{n_{i},k_{i}}\not \in U_{\beta _{k_{i}}}$. It follows from Lemma \ref{GKL1} that there is $\gamma \in \mathbb{N}^{\mathbb{N}}$ such that $\beta _{k_{i}}\leq \gamma$ for every $i\in \mathbb{N}$. Therefore, for any $i\in \mathbb{N}$, $x_{n_{i},k_{i}}\not \in U_{\gamma}$. We conclude that the sequence $\{x_{n_{i},k_{i}}\}_{i\in \mathbb{N}}$ does not converge to $0$ and this is a contradiction.

Therefore, we can find a minimal natural number $k_{\alpha}$ such that $D_{k_{\alpha}}(\alpha)$ is a neighborhood of $\pi (0)$ for each $\alpha \in \mathbb{N}^{\mathbb{N}}$. It is clear that $D_{k_{\alpha}}(\alpha)\subset U_{\alpha}$. Moreover, for $i\in \mathbb{N}$, fix $\alpha ^{(i)}=(i,\alpha_{2},\alpha_{3},...)\in \mathbb{N}^{\mathbb{N}}$. Then for any $\beta =(\beta_{1},\beta_{2},...)\in I_{1}(\alpha ^{(i)})$, $D_{1}(\beta)=D_{1}(\alpha ^{i})$. Therefore, $\{D_{1}(\alpha):\alpha \in \mathbb{N}^{\mathbb{N}}\}=\{D_{1}(\alpha^{(i)}):i\in \mathbb{N}\}$ is countable. So, $\{D_{k}(\alpha):k\in \mathbb{N},\alpha \in \mathbb{N}^{\mathbb{N}}\}$ is countable. Furthermore, $\{D_{k_{\alpha}}(\alpha):\alpha \in \mathbb{N}^{\mathbb{N}}\}\subset \{D_{k}(\alpha):k\in \mathbb{N},\alpha \in \mathbb{N}^{\mathbb{N}}\}$. Therefore, the countability of the family $\{D_{k_{\alpha}}(\alpha):\alpha \in \mathbb{N}^{\mathbb{N}}\}$ is obtained. In conclusion, the family $\{int(D_{k_{\alpha}}(\alpha)):\alpha \in \mathbb{N}^{\mathbb{N}}\}$ is a countable base of open neighborhoods at $\pi (0)$. It follows from the homogeneity of $G/H$ of Lemma \ref{yl1} that $G$ is first-countable.
\end{proof}

By Lemma \ref{yl1}, we know that the quotient space $G/H$ is homogenous, so it is clear that if $G/H$ is first-countable, then it has an $\omega^{\omega}$-base. Therefore, by Theorem \ref{h} and Theorem \ref{fdl1}, we have the following corollary.

\begin{corollary}\label{3tl1}
Suppose that $(G,\tau ,\oplus)$ is a strongly topological gyrogroup with a symmetric neighborhood base $\mathscr U$ at $0$ and $H$ is an admissible subgyrogroup generated from $\mathscr U$. Suppose further that $H$ is neutral, then $G/H$ is metrizable if and only if $G/H$ is Fr\'echet-Urysohn with an $\omega^{\omega}$-base.
\end{corollary}

\section{cardinality of quotient spaces}

We denote by $\chi (X)$, $\pi \chi (X)$, $\omega (X)$ and $\pi \omega (X)$ the character, $\pi$-character, weight and $\pi$-weight of a space $X$, respectively. Similarly, $\chi (x,X)$ and $\pi \chi (x,X)$ stand for the character and $\pi$-character of $X$ at a point $x\in X$. In this section, we show that if $(G,\tau ,\oplus)$ is a strongly topological gyrogroup with a symmetric neighborhood base $\mathscr U$ at $0$, $H$ is an admissible subgyrogroup generated from $\mathscr U$ and $H$ is neutral, then $\pi\chi(G/H)=\chi(G/H)$ and $\pi\omega (G/H)=\omega (G/H)$.

\begin{theorem}\label{3yl1}
Let $(G,\tau ,\oplus)$ be a strongly topological gyrogroup with a symmetric neighborhood base $\mathscr U$ at $0$ and $H$ an admissible subgyrogroup generated from $\mathscr U$. If $H$ is neutral, then $\pi\chi(G/H)=\chi(G/H)$.
\end{theorem}

\begin{proof}
Let $\pi$ be the natural homomorphism from $G$ to $G/H$ and $e^{*}=\pi(0)$ in $G/H$. It follows from Lemma \ref{yl1} that the quotient space $G/H$ is homogenous, so we have $\chi(G/H)=\chi(e^{*},G/H)$ and $\pi\chi(G/H)=\pi\chi(e^{*},G/H)$. Then it suffices to show $$\chi(e^{*},G/H)\leq \pi \chi(e^{*},G/H).$$

Let $\mathcal{B}$ be a local $\pi$-base for $G/H$ at $e^{*}$ such that $|\mathcal{B}|=\pi\chi(e^{*},G/H)$. For every $V\in \mathcal{B}$, put $W_{V}=\pi^{-1}(V)$. We show that the family $\mathcal{C}=\pi (W_{V}\oplus (\ominus W_{V}))$ is a local base for $G/H$ at $e^{*}$. Assume that $O$ is an open neighborhood of $e^{*}$ in $G/H$. Since $H$ is an admissible subgyrogroup generated from $\mathscr U$, there exists a sequence $\{U_{n}:n\in \mathbb{N}\}$ of symmetric open neighborhoods of the identity $0$ in $G$ satisfying $U_{n}\in \mathscr U$ and $U_{n+1}\oplus (U_{n+1}\oplus U_{n+1})\subset U_{n}$, for each $n\in \mathbb{N}$, and such that $H=\bigcap _{n\in \mathbb{N}}U_{n}$. Therefore, we have $gyr[x,y](H)=H$ for any $x,y\in G$. Since $H$ is neutral, we choose $U,W\in \mathscr{U}$ such that $\pi(U\oplus U)\subset O$, $W\subset U$, and $H\oplus W\subset U\oplus H$. As $\mathcal{B}$ is a local $\pi$-base at $e^{*}$ in $G/H$, there exists $V\in \mathcal{B}$ such that $V\subset \pi(W)$. Then $W_{V}=\pi^{-1}(V)\subset \pi^{-1}\pi(W)=W\oplus H$.

\begin{eqnarray}
W_{V}\oplus (\ominus W_{V})&\subset&(W\oplus H)\oplus \ominus (W\oplus H)\nonumber\\
&=&(W\oplus H)\oplus gyr[W,H](\ominus H\ominus W)\nonumber\\
&=&(W\oplus H)\oplus gyr[W,H](H\oplus W)\nonumber\\
&=&(W\oplus H)\oplus (H\oplus W)\nonumber\\
&=&W\oplus (H\oplus gyr[H,W](H\oplus W))\nonumber\\
&=&W\oplus (H\oplus (H\oplus W))\nonumber\\
&=&W\oplus ((H\oplus H)\oplus gyr[H,H](W))\nonumber\\
&=&W\oplus (H\oplus W)\nonumber\\
&\subset& W\oplus (U\oplus H)\nonumber\\
&\subset& U\oplus (U\oplus H)\nonumber\\
&=&(U\oplus U)\oplus gyr[U,U](H)\nonumber\\
&=&(U\oplus U)\oplus H.\nonumber
\end{eqnarray}

Therefore, $\pi (W_{V}\oplus (\ominus W_{V}))\subset \pi((U\oplus U)\oplus H)=\pi(U\oplus U)\subset O$. It is shown that the family $\mathcal{C}=\pi (W_{V}\oplus (\ominus W_{V}))$ is a local base for $G/H$ at $e^{*}$ and it is clear that $|\mathcal{C}|\leq |\mathcal{B}|$. Then $\chi(e^{*},G/H)\leq \pi \chi(e^{*},G/H)$.
\end{proof}

\begin{proposition}\label{3mt1}
Let $(G,\tau ,\oplus)$ be a strongly topological gyrogroup with a symmetric neighborhood base $\mathscr U$ at $0$ and $H$ an admissible subgyrogroup generated from $\mathscr U$. If $H$ is neutral, $B$ is a dense subset of $G/H$, and $A\subset G$ such that $\pi(A)=B$, then $\pi(A\oplus U)=G/H$, for every neighborhood $U$ of the identity element in $G$.
\end{proposition}

\begin{proof}
Since $H$ is an admissible subgyrogroup generated from $\mathscr U$, there exists a sequence $\{U_{n}:n\in \mathbb{N}\}$ of symmetric open neighborhoods of the identity $0$ in $G$ satisfying $U_{n}\in \mathscr U$ and $U_{n+1}\oplus (U_{n+1}\oplus U_{n+1})\subset U_{n}$, for each $n\in \mathbb{N}$, and such that $H=\bigcap _{n\in \mathbb{N}}U_{n}$. Therefore, we have $gyr[x,y](H)=H$ for any $x,y\in G$. For arbitrary open neighborhood $U$ of $0$ in $G$, there exists $V\in \mathscr{U}$ such that $V\oplus V\subset U$. Since $H$ is neutral, there exists $V\in \mathscr{U}$ such that $H\oplus V\subset U\oplus H$. Assume that $B$ is a dense subset of $G/H$, and $A\subset G$ such that $\pi(A)=B$. If $y\in G/H$, there exists $x\in G$ such that $\pi(x)=y$. Since $B$ is dense in $G/H$ and $\pi (x\oplus V)$ is open in $G/H$, it is clear that $B\cap \pi(x\oplus V)\not=\emptyset$, that is, there exists $b\in B$ and $b\in \pi(x\oplus V)$. Then we can find $a\in G$ satisfying $\pi(a)=b$. Then $\pi (a)\in \pi(x\oplus V)$, that is, $a\in (x\oplus V)\oplus H$. Therefore, there exist $v\in V$ and $h\in H$ such that $a=(x\oplus v)\oplus h$. Since

\begin{eqnarray}
(x\oplus v)&=&((x\oplus v)\oplus h)\oplus gyr[x\oplus v,h](\ominus h)\nonumber\\
&=&a\oplus gyr[x\oplus v,h](\ominus h),\nonumber
\end{eqnarray}

we have that

\begin{eqnarray}
x&=& (x\oplus v)\oplus gyr[x,v](\ominus v)\nonumber\\
&=&(a\oplus gyr[x\oplus v,h](\ominus h))\oplus gyr[x,v](\ominus v)\nonumber\\
&\in&(a\oplus gyr[x\oplus v,h](H))\oplus gyr[x,v](V)\nonumber\\
&=&(a\oplus H)\oplus V\nonumber\\
&=&a\oplus (H\oplus gyr[H,a](V))\nonumber\\
&=&a\oplus (H\oplus V)\nonumber\\
&\subset&a\oplus (U\oplus H)\nonumber\\
&\subset&(a\oplus U)\oplus gyr[a,U](H)\nonumber\\
&=&(a\oplus U)\oplus H.\nonumber
\end{eqnarray}

Therefore, $y=\pi(x)\in \pi((a\oplus U)\oplus H)=\pi(a\oplus H)\subset \pi(A\oplus U)$. We conclude that $\pi(A\oplus H)=G/H$.
\end{proof}

\begin{theorem}
Let $(G,\tau ,\oplus)$ be a strongly topological gyrogroup with a symmetric neighborhood base $\mathscr U$ at $0$ and $H$ an admissible subgyrogroup generated from $\mathscr U$. If $H$ is neutral, then $\pi\omega (G/H)=\omega (G/H)$.
\end{theorem}

\begin{proof}
Since $H$ is an admissible subgyrogroup generated from $\mathscr U$, there exists a sequence $\{U_{n}:n\in \mathbb{N}\}$ of symmetric open neighborhoods of the identity $0$ in $G$ satisfying $U_{n}\in \mathscr U$ and $U_{n+1}\oplus (U_{n+1}\oplus U_{n+1})\subset U_{n}$, for each $n\in \mathbb{N}$, and such that $H=\bigcap _{n\in \mathbb{N}}U_{n}$. Therefore, we have $gyr[x,y](H)=H$ for any $x,y\in G$. Put $\kappa =\pi\omega (G/H)$ and $B$ is a dense subset of $G/H$ with $|B|\leq \kappa$. Set $A\subset G$ with $\pi (A)=B$ and $|A|=|B|$. As $\pi\chi (G/H)\leq \pi\omega (G/H)$, by Theorem \ref{3yl1}, there exists a family $\gamma$ of symmetric open neighborhoods of $0$ in $G$ with $|\gamma|\leq \kappa$ and $\{\pi (U):U\in \gamma\}$ is a local base at $\pi (0)$ in $G/H$. Put $\mathcal{C}=\{\pi (a\oplus W):a\in A,W\in \gamma\}$. Then $\mathcal{C}$ is a base for $G/H$.

Indeed, for arbitrary $y\in G/H$, choose $x\in G$ with $\pi (x)=y$. Let $O$ be an open neighborhood of $y$ in $G/H$, then we can find $U\in \mathscr{U}$ such that $\pi (x\oplus (U\oplus U))\subset O$. Then $\pi (x\oplus (V\oplus V))\subset \pi (x\oplus (U\oplus U))$ for some $V\in \gamma$. Since $H$ is neutral, there exists $W\in \mathscr{U}$ with $H\oplus W\subset V\oplus H$. By Proposition \ref{3mt1}, there exists $a\in A$ such that $y=\pi (x)\in \pi (a\oplus W)$, then $x\in (a\oplus W)\oplus H$. We can find $w\in W$ and $h\in H$ such that $x=(a\oplus w)\oplus h$. Then $x\oplus gyr[a\oplus w,h](\ominus h)=((a\oplus w)\oplus h)\oplus gyr[a\oplus w,h](\ominus h)=a\oplus w$. Moreover,
\begin{eqnarray}
a&=&(a\oplus w)\oplus gyr[a,w](\ominus w)\nonumber\\
&=&(x\oplus gyr[a\oplus w,h](\ominus h))\oplus gyr[a,w](\ominus w)\nonumber\\
&\in& (x\oplus gyr[a\oplus w,h](H))\oplus gyr[a,w](W)\nonumber\\
&=&(x\oplus H)\oplus W.\nonumber
\end{eqnarray}

Thus, there exist $h_{1}\in H$ and $w_{1}\in W$ such that $a=(x\oplus h_{1})\oplus w_{1}$, so

\begin{eqnarray}
x&=&(a\oplus w)\oplus h\nonumber\\
&=&(((x\oplus h_{1})\oplus w_{1})\oplus w)\oplus h\nonumber\\
&=&((x\oplus (h_{1}\oplus gyr[h_{1},x](w_{1})))\oplus w)\oplus h\nonumber\\
&=&((x\oplus (h_{1}\oplus w_{2}))\oplus w)\oplus h\nonumber ~~( for~some~w_{2}\in W,~since~W\in  \mathscr{U})
\end{eqnarray}

Since $H\oplus W\subset V\oplus H$, we can find $v\in V$ and $h_{2}\in H$ such that $h_{1}\oplus w_{2}=v\oplus h_{2}$. Then

\begin{eqnarray}
((x\oplus (h_{1}\oplus w_{2}))\oplus w)\oplus h&=&((x\oplus (v\oplus h_{2}))\oplus w)\oplus h\nonumber\\
&=&(((x\oplus v)\oplus gyr[x,v](h_{2}))\oplus w)\oplus h\nonumber\\
&=&(((x\oplus v)\oplus h_{3})\oplus w)\oplus h\nonumber ~~(H~is~generated~from~\mathscr{U})\\
&=&((x\oplus v)\oplus (h_{3}\oplus gyr[h_{3},x\oplus v](w)))\oplus h\nonumber\\
&=&((x\oplus v)\oplus (h_{3}\oplus w_{3}))\oplus h\nonumber ~~(since~w\in W)\\
&=&((x\oplus v)\oplus (v_{1}\oplus h_{4}))\oplus h\nonumber\\
&=&(((x\oplus v)\oplus v_{1})\oplus gyr[x\oplus v,v_{1}](h_{4}))\oplus h\nonumber\\
&=&((x\oplus (v\oplus gyr[v,x](v_{1})))\oplus h_{5})\oplus h\nonumber\\
&\in& ((x\oplus (V\oplus V))\oplus H)\oplus H.\nonumber
\end{eqnarray}

Thus,
\begin{eqnarray}
y&=&\pi (x)\nonumber\\
&\in&\pi (a\oplus W)\nonumber\\
&\subset& \pi (((x\oplus (V\oplus V))\oplus H)\oplus H)\nonumber\\
&=&\pi((x\oplus (V\oplus V))\oplus H)\nonumber\\
&=&\pi (x\oplus (V\oplus V))\nonumber\\
&\subset&\pi (x\oplus (U\oplus U))\nonumber\\
&\subset&O.\nonumber
\end{eqnarray}

Moreover, since $|\mathcal{C}|\leq |A|\cdot |\gamma|\leq \kappa$, we have that $\omega (G/H)\leq \pi \omega (G/H)$.
\end{proof}

Naturally, we pose the following question.

\begin{question}
If $(G,\tau ,\oplus)$ is a strongly topological gyrogroup with a symmetric neighborhood base $\mathscr U$ at $0$ and $H$ is an admissible subgyrogroup generated from $\mathscr U$, whether could the two conditions $\pi\omega (G/H)=\omega (G/H)$ and $\pi\chi(G/H)=\chi(G/H)$ hold?
\end{question}


\begin{thebibliography}{33}

\bibitem{AA} A.V. Arhangel' ski\v\i, M. Tkachenko, {\it Topological Groups and Related Structures}, Atlantis Press and World Sci., 2008.

\bibitem{AA2010} A.V. Arhangel' ski\v\i, M.M. Choban, {\it On remainders of rectifiable spaces}, Topol. Appl., 157 (2010) 789-799.

\bibitem{AW} W. Atiponrat, {\it Topological gyrogroups: generalization of topological groups}, Topol. Appl., 224(2017)73--82.

\bibitem{BL} M. Bao, F. Lin, {\it Feathered gyrogroups and gyrogroups with countable pseudocharacter}, Filomat, 33 (16)(2019) 5113-5124.

\bibitem{BL1} M. Bao, F. Lin, {\it Submetrizability of strongly topological gyrogroups}, Houston Jour. Math., (2020) Accepted.

\bibitem{BL2} M. Bao, F. Lin, {\it Quotient with respect to admissible L-subgyrogroups}, Topol. Appl., (2020) 107492.

\bibitem{BL3} M. Bao, F. Lin, {\it Submaximal properties in (strongly) topological gyrogroups}, Filomat, (2020) Accepted.

\bibitem{BLX} M. Bao, X. Ling, X. Xu, {\it Strongly topological gyrogroups and quotient with respect to L-subgyrogroups}, Houston Jour. Math., (2021) Accepted.

\bibitem{BT} T. Banakh, {\it Topological spaces with an $\omega^{\omega}$-base}, Diss. Math., 538 (2019) 1-141.

\bibitem{BZX} M. Bao, X. Zhang, X. Xu, {\it Separability in (strongly) topological gyrogroups}, https://arxiv.org/abs/2011.02633.

\bibitem{BZX1} M. Bao, X. Zhang, X. Xu, {\it The strong Pytkeev property and strongly countable completeness in (strongly) topological gyrogroups}, Filomat, (2021) Accepted.

\bibitem{BZX2} M. Bao, X. Zhang, X. Xu, {\it Topological gyrogroups with Fr\'echet-Urysohn property and $\omega^{\omega}$-base}, Bulle. Iranian Math. Sco., (2021) DOI:10.1007/s41980-021-00576-w.

\bibitem{CZ} Z. Cai, S. Lin, W. He, {\it A note on Paratopological Loops}, Bulletin of the Malaysian Math. Sci. Soc., 42(5) (2019) 2535-2547.

\bibitem{CZ1} Z. Cai, P. Ye, S. Lin, B. Zhao, {\it A note on paratopological groups with an $\omega ^{\omega}$-base}, Topol. Appl., 275 (2020) 107151.

\bibitem{E} R. Engelking, General Topology(revised and completed edition), Heldermann Verlag,
Berlin, 1989.

\bibitem{FM} M. Ferreira, {\it Factorizations of M\"{o}bius gyrogroups}, Adv. Appl. Clifford Algebras, 19 (2009) 303--323.

\bibitem{FM1} M. Ferreira, G. Ren, {\it M\"{o}bius gyrogroups: A Clifford algebra approach}, J. Algebra, 328 (2011) 230-253.

\bibitem{FM2} M. Ferreira, S. Teerapong, {\it Orthogonal gyrodecompositions of real inner product gyrogroups}, Symmetry, 12(6) (2020),941, 37pp.

\bibitem{FS} S.P. Franklin, {\it Spaces in which sequences suffice}, Fundam. Math. 57 (1965) 107--115.

\bibitem{GK} S. Gabriyelyan, J. Kakol, {\it On topological spaces and topological groups with certain local countable networks}, Topol. Appl., 190 (2015) 59-73.

\bibitem{GKL} S. Gabriyelyan, J. Kakol, A. Leiderman,{\it On topological groups with a small base and metrizability}, Fundam. Math., 299 (2015) 129-158.

\bibitem{LPT} A. Leiderman, V. Pestov, A. Tomita, {\it On topological groups admitting a base at identity indexed with $\omega^{\omega}$}, Fundam. Math., 238 (2017) 79-100.

\bibitem{LF} F. Lin. R. Shen, {\it On rectifiable spaces and paratopological groups}, Topol. Appl., 158(2011)597--610.

\bibitem{LF1} F. Lin. C. Liu, S. Lin, {\it A note on rectifiable spaces}, Topol. Appl., 159(2012)2090--2101.

\bibitem{LF2} F. Lin, {\it Compactly generated rectifiable spaces or paratopological groups}, Math. Commun., 18(2013)417--427.

\bibitem{linbook} S. Lin, Z. Yun, {\it Generalized Metric Spaces and Mappings},
Science Press, Atlantis Press, 2017.

\bibitem{ST} T. Suksumran, K. Wiboonton, {\it Isomorphism theorems for gyrogroups and L-subgyrogroups}, J. Geom. Symmetry Phys., 37 (2015) 67--83.

\bibitem{UA1988} A.A. Ungar, {\it The Thomas rotation formalism underlying a nonassociative group structure for relativistic velocities}, Appl. Math. Lett., 1(4) (1988) 403--405.

\bibitem{UA2002} A.A. Ungar,{\it  Beyond the Einstein addition law and its gyroscopic Thomas precession: The theory
of gyrogroups and gyrovector spaces}, Fundamental Theories of Physics, vol. 117, Springer, Netherlands, 2002.

\bibitem{UA} A.A. Ungar, {\it Analytic Hyperbolic Geometry and Albert Einstein's Special Theory of Relativity}, World Scientific, Hackensack, New Jersey, 2008.

\bibitem{UA2005} A.A. Ungar, {\it Analytic hyperbolic geometry: Mathematical foundations and applications}, World Scientific, Hackensack, 2005.

\end{thebibliography}
\end{document}